\documentclass{amsart}

\usepackage{graphicx,epsfig}

\usepackage{amsmath,amssymb,latexsym, amsfonts, amscd, amsthm, xy}

\usepackage{url}

\usepackage{tikz}

\usepackage{hyperref}

\input{xy}
\xyoption{all}

\makeindex 

=630pt \headheight = 12pt \headsep = 20pt

\theoremstyle{plain}
\newtheorem{theorem}{Theorem}[section]

\newtheorem{corollary}[theorem]{Corollary}

\newtheorem{lemma}[theorem]{Lemma}

\newtheorem{question}[theorem]{Question}

\theoremstyle{definition}

\newtheorem{remark}[theorem]{Remark}

\def\e{\mathbf{e}}

\def\bF{\mathbb{F}}

\def\bK{\mathbb{K}}

\def\bL{\mathbb{L}}

\def\bP{\mathbb{P}}
\def\bQ{\mathbb{Q}}
\def\bZ{\mathbb{Z}}

\def\C{\mathbf{C}}
\def\D{\mathbf{D}}

\def\I{\mathbf{I}}

\def\cA{\mathcal{A}}
\def\cB{\mathcal{B}}

\def\cD{\mathcal{D}}
\def\cE{\mathcal{E}}

\def\cH{\mathcal{H}}

\def\cK{\mathcal{K}}

\def\cM{\mathcal{M}}

\def\cO{\mathcal{O}}

\def\cP{\mathcal{P}}

\def\cS{\mathcal{S}}

\def\cX{\mathcal{X}}

\def\fB{\mathfrak{B}}

\def\fQ{\mathfrak{Q}}

\def\fp{\mathfrak{p}}

\def\fq{\mathfrak{q}}

\def\deg{\mathrm{deg}}

\def\Gal{\mathrm{Gal}}

\def\GL{\mathrm{GL}}

\def\Card{\mathbf{Card}}

\def\alg{\mathbf{alg}}

\def\e{\mathbf{e}}

\def\Dr{\mathbf{DR}}

\begin{document}

\title{A Carlitz module analogue of the Grunwald--Wang theorem}

\author{Nguyen Ngoc Dong Quan}

\date{July 27, 2014}

\address{Department of Mathematics \\
         University of British Columbia \\
         Vancouver, British Columbia \\
         V6T 1Z2, Canada}

\email{\href{mailto:dongquan.ngoc.nguyen@gmail.com}{\tt dongquan.ngoc.nguyen@gmail.com}}

\maketitle

\tableofcontents

\section{Introduction}

The classical Grunwald--Wang theorem is an example of a local--global (or Hasse) principle stating that except in some \textit{special} cases which are precisely determined, an element $m$ in a number field $\bK$ is an $a$-th power in $\bK$ if and only if it is an $a$-th power in the completion $\bK_{\wp}$ of $\bK$ for all but finitely many primes $\wp$ of $\bK$. In more concrete terms, the equation $x^a = m$ has a solution $x_{\bK}$ in $\bK$ if and only if it has a solution $x_{\wp}$ in the completion $\bK_{\wp}$ for all but finitely many primes $\wp$ of $\bK$. The aim of this article is to prove a Carlitz module analogue of the Grunwald--Wang theorem.

Throughout the paper, let $A = \bF_q[T], k = \bF_q(T)$, where $q$ is a power of a prime $p$, and let $k^{\alg}$ be the algebraic closure of $k$. Let $\tau$ be the map defined by $\tau(x) = x^q$. Let $k\langle \tau \rangle $ denote the twisted polynomial ring equipped with twisted multiplication, namely, $\tau a = a^q \tau$ for all $a \in k$. Let $C$ be the Carlitz module, i.e., $C : A \rightarrow k\langle \tau \rangle$ be an $\bF_q$-algebra homomorphism such that $C_T = T + \tau$. Note that for every commutative $k$-algebra $\D$, the action of $\sum_{j}a_i \tau^j \in k\langle \tau \rangle$ on an element $d \in \D$ is defined by
\begin{align*}
(\sum_{j}a_i \tau^j)(d) = \sum_{j}a_i d^{q^j}.
\end{align*}

A special case of an analogue of the Grunwald--Wang theorem for function fields was proved by G.-J. van der Heiden \cite{van-der-Heiden}. We now describe the result of van der Heiden. Let $\cX$ be a projective, smooth, geometrically irreducible curve over the finite field $\bF_q$, and let $\infty$ be a fixed closed point on $\cX$. Let $\bF_q(\cX)$ be the function field of $\cX$, and let $\cO_{\bF_q(\cX)}$ be the ring of functions in $\bF_q(\cX)$ that are regular outside $\infty$. Let $\cK$ be a finite separable extension of $\bF_q(\cX)$, and let $\iota : \cO_{\bF_q(\cX)} \rightarrow \cK$ be the natural embedding of $\cO_{\bF_q(\cX)}$ into $\cK$. Let $\Dr : \cO_{\bF_q(\cX)} \rightarrow \cK\langle \tau \rangle$ be a Drinfeld module over $\cK$ of rank 1, that is, $\Dr$ is an $\bF_q$-algebra homomorphism such that for all $f \in \cO_{\bF_q(\cX)}$, we have
\begin{align*}
\Dr_f = \sum_{j = 0}^{\deg(f)}f_i\tau^i,
\end{align*}
where the $f_i$ are elements in $\cK$ with $f_{\deg(f)} \in \cK^{\times}$ and $f_0 = \iota(f)$. In \cite{van-der-Heiden}, Van der Heiden proved the following result.
\begin{theorem}
\label{Theorem-Van-der-Heiden-theorem}
$(\text{van der Heiden, \cite[Proposition {\bf5}]{van-der-Heiden}})$

Let $P$ be a prime element of $\cO_{\bF_q(\cX)}$, and let $m$ be an element in $\cK$. Then the equation $\Dr_P(x) = m$ has a solution $x_{\cK}$ in $\cK$ if and only if it has a solution $x_{\wp}$ in the completion $\cK_{\wp}$ for every place $\wp$ of $\cK$.

\end{theorem}

Since $P$ is a prime element of $\cO_{\bF_q(\cX)}$ and the equation $\Dr_P(x) = m$ in Theorem \ref{Theorem-Van-der-Heiden-theorem} is considered over the field of definition of $\Dr$, namely $\cK$, we see that Theorem \ref{Theorem-Van-der-Heiden-theorem} might be viewed as analogous to the following special case of the classical Grunwald--Wang theorem: \textit{the equation $x^p = m$ with prime $p$ in $\bZ$ and $m \in \bQ$ has a solution $x_{\bQ}$ in $\bQ$ if and only if it has a solution $x_l$ in the $l$-adic field $\bQ_l$ for every prime $l$.}

To better understand the analogies between Theorem \ref{Theorem-Van-der-Heiden-theorem} and the above statement, one can consider the case when $\cX = \bP_1/\bF_q$. Thus $\cK = k, \cO_{\bF_q(\cX)} = A$, and the Drinfeld module $\Dr$ is the Carlitz module $C$ that was introduced earlier. There are well-known analogies \cite{Carlitz} \cite{Goss} \cite{Hayes} \cite{Rosen} \cite{Thakur-book} between the map $x \mapsto x^a$ with $a \in \bZ$ and the map $x \mapsto C_a(x)$ with $a \in A$; for example, adjoining torsion points of the Carlitz module to the field $k$ generates cyclotomic function fields in a similar manner as cyclotomic extensions of $\bQ$ obtained by adjoining roots of unity to $\bQ$. Hence Theorem \ref{Theorem-Van-der-Heiden-theorem} can be regarded as a special case of the classical Grunwald--Wang theorem.

In the hope of obtaining the best possible analogy with the classical Grunwald--Wang theorem, we will have to be content with working with the Carlitz module over $k$ because, among other reasons, the ring $\cO_{\bF_q(\cX)}$ might have non-principal ideals and the analogies between $(A, k)$ and $(\bZ, \bQ)$ are stronger. It is natural to search for a generalization of Theorem \ref{Theorem-Van-der-Heiden-theorem} that might be viewed as the full Grunwald--Wang theorem in the function field setting.
\begin{question}
\label{Question-An-analogue-of-the-Grunwald-Wang-theorem-for-function-fields}

Let $\bK$ be a finite separable extension of $k$. Let $a$ be an element in $A$ of positive degree, and let $m$ be an element of $\bK$. Consider the equation
\begin{align*}
C_a(x) = m.
\end{align*}
If $C_a(x) = m$ has a solution $x_{\wp}$ in the completion $\bK_{\wp}$ for all but finitely many primes $\wp$ of $\bK$, is it true that the equation $C_a(x) = m$ has a solution $x_{\bK}$ in $\bK$?

\end{question}

An affirmative answer to Question \ref{Question-An-analogue-of-the-Grunwald-Wang-theorem-for-function-fields} should give rise to a result that is analogous to the classical Grunwald--Wang theorem in the number field case, and is a generalization of the result of van der Heiden in the Carlitz module setting. Indeed, in contrast to Theorem \ref{Theorem-Van-der-Heiden-theorem}, although the field of definition of the Carlitz module is $k$, we consider the equation $C_a(x) = m$ over any finite separable extension of $k$. Furthermore $a$ can be any element in $A$ of positive degree, and we \textit{only} require that the equation $C_a(x) = m$ is locally solvable at \textit{all but finitely many primes} of $\bK$. In comparison with the number field case, the classical Grunwald--Wang theorem considers global solvability of the equation $x^a = m$ over a number field although the power map $x^a$ with $a \in \bZ$ is defined over $\bQ$.

Let $\Lambda_a \subset k^{\alg}$ be the cyclic $A$-module defined by $\Lambda_{a} = \left\{\lambda \in k^{\alg} \; | \; C_a(x) = 0 \right\}$. The following theorem is a summary of the main results in this paper.
\begin{theorem}
\label{Theorem-A-summary-of-the-main-results}

Question \ref{Question-An-analogue-of-the-Grunwald-Wang-theorem-for-function-fields} has an affirmative answer in the following cases:
\begin{itemize}

\item [(i)] $\bK$ contains a primitive generator of the cyclic $A$-module $\Lambda_a$.

\item [(ii)] $\bK$ is a Galois extension of $k$ such that every prime $P$ dividing $a$ is unramified in $\bK$.

\end{itemize}

\end{theorem}

Part $(i)$ of Theorem \ref{Theorem-A-summary-of-the-main-results} is the content of Corollary \ref{Corollary-A-special-case-of-a-Carlitz-module-analogue-of-the-Grunwald-Wang-theorem}, and part $(ii)$ is the content of Theorem \ref{Theorem-A-Carlitz-module-analogue-of-the-Grunwald-Wang-theorem}. Note that Theorem \ref{Theorem-A-Carlitz-module-analogue-of-the-Grunwald-Wang-theorem} and Corollary \ref{Corollary-A-special-case-of-a-Carlitz-module-analogue-of-the-Grunwald-Wang-theorem} actually are more general than Theorem \ref{Theorem-A-summary-of-the-main-results}. More precisely, in the statements of Theorem \ref{Theorem-A-Carlitz-module-analogue-of-the-Grunwald-Wang-theorem} and Corollary \ref{Corollary-A-special-case-of-a-Carlitz-module-analogue-of-the-Grunwald-Wang-theorem}, instead of assuming local solvability of the equation $C_a(x) = m$ at almost all primes in $\bK$, we only assume that $C_a(x) = m$ has a solution $x_{\wp}$ in the completion $\bK_{\wp}$ for all primes $\wp$ in a set of primes of $\bK$ of Dirichlet density greater than $1 - 1/\phi(a)q^{\deg(a)}$. Here $\phi(a)$ is the function field analogue of the Euler $\phi$-function.

Let us say a few words about the ideas of the proof of Theorem \ref{Theorem-A-summary-of-the-main-results}. In \cite{Ono-Terasoma}, T. Ono and T. Terasoma attempted to give a different proof of a special case of the classical Grunwald--Wang theorem. Unfortunately the statement of a special case of the Grunwald--Wang theorem as well as its attempted proof given in \cite{Ono-Terasoma} were erroneous as pointed out by T. Ono \cite{Ono}. Despite the errors in \cite{Ono-Terasoma}, the underlying ideas in that paper can be modified to apply to the function field setting. We will carry out such modifications in the proof of Theorem \ref{Theorem-A-summary-of-the-main-results} whose main ingredients are the use of Tchebotarev density theorem and basic facts in class field theory over function fields. Thus the ideas of this paper are completely different from those used in \cite{van-der-Heiden} in which G.-J. van der Heiden extensively used Galois cohomology.

\section{The splitting field of the polynomial $C_a(x) - m$}
\label{Section-The-splitting-field-of-the-polynomial-C-a(x)-=-m}

Let $\bK$ be a finite, separable extension of $k$. Let $a$ be an element in $A$, and let $m$ be an element in $\bK$. In this section, we will describe the splitting field of the polynomial $C_a(x) - m \in \bK[x]$ over $\bK$ when the equation $C_a(x) = m$ is locally solvable at all primes in a set of primes of $\bK$ of Dirichlet density greater than $1 - 1/\phi(a)q^{\deg(a)}$. The case that the equation $C_a(x) = m$ has a solution $x_{\wp}$ in $\bK_{\wp}$ for almost all primes in $\bK$ follows immediately as a special case. The following result is very elementary and can be derived from the results in \cite[Chapter {\bf12}]{Rosen}. For the sake of self-containedness, we also include a proof here.

\begin{lemma}
\label{Lemma-Description-of-elements-in-C-a-lambda}

Let $a$ be a polynomial of positive degree $s$ in $A$. Let $\Lambda_a \subset k^{\alg}$ be the $A$-module defined by
\begin{align*}
\Lambda_a = \left\{\lambda \in k^{\alg} \; | \; C_a(\lambda) = 0 \right\},
\end{align*}
and let $\lambda_a$ be a generator for the cyclic $A$-module $\Lambda_a$ $(\text{see \cite[Chapter {\bf12}]{Rosen}})$. Then
\begin{itemize}

\item [(i)] $C_b(\lambda_a) \ne C_c(\lambda_a)$ for every $b, c \in A$ with $b \ne c$ and $\max(\deg(b), \deg(c)) < s$; and

\item [(ii)] $\Lambda_a = \{C_b(\lambda_a) \; | \; \text{$b \in A$ with $\deg(b) < s$}\}$.

\end{itemize}

\end{lemma}

\begin{proof}

Part $(ii)$ follows immediately from part $(i)$. We now prove part $(i)$. Let $b, c$ be distinct elements in $A$ with $\max(\deg(b), \deg(c)) < s$. Assume that $C_b(\lambda_a) = C_c(\lambda_a)$. It then follows that
\begin{align*}
C_{b - c}(\lambda_a) = 0.
\end{align*}
Since $b \ne c$, we see that $b - c \ne 0$, and hence $0 \le \deg(b - c) < s$. If $\deg(b - c) = 0$, we see that $b - c$ belongs to $\bF_q^{\times}$, and hence $C_{b - c}(\lambda_a) = (b - c)\lambda_a = 0$. Thus $\lambda_a = 0$, which is a contradiction.

If $\deg(b - c) > 0$, we know from \cite[Proposition {\bf12.4}]{Rosen} that $\Lambda_{b - c}$ has exactly $q^{\deg(b - c)}$ elements. Since $C_{b - c}(\lambda_a) = 0$ and $\lambda_a$ is a generator of the $A$-module $\Lambda_a$, we deduce that $\Lambda_a \subseteq \Lambda_{b - c}$. Thus
\begin{align*}
q^s = q^{\deg(a)} = \#\Lambda_a \le \#\Lambda_{b - c} = q^{\deg(b - c)},
\end{align*}
and therefore $s \le \deg(b - c)$, which is a contradiction. Therefore $C_b(\lambda_a) \ne C_c(\lambda_a)$.

\end{proof}

We now prove the main result in this section.

\begin{theorem}
\label{Theorem-The-description-of-the-splitting-field-of-the-Carlitz-equation}

Let $\bK$ be a finite, separable extension of $k$. Let $m$ be an element in $\bK$, and let $a$ be a polynomial of positive degree in $A$. Let $\bL$ be the splitting field of the polynomial $C_a(x) - m \in \bK[x]$ over $\bK$. Let $\lambda_a$ be a generator of the $A$-module $\Lambda_a$, where
\begin{align*}
\Lambda_a = \left\{\lambda \in k^{\alg} \; | \; C_a(\lambda) = m \right\}.
\end{align*}
Assume that the equation $C_a(x) = m$ has a solution $x_{\wp}$ in the completion $\bK_{\wp}$ for all but finitely many primes $\wp$ of $\bK$. Then $\bL = \bK(\lambda_a)$.

\end{theorem}

\begin{proof}

Note that $\bK \subset k^{\alg}$ since $\bK$ is a finite extension of $k$. Let $h$ be an element in $k^{\alg}$ such that
\begin{align*}
C_a(h) = m.
\end{align*}
By Lemma \ref{Lemma-Description-of-elements-in-C-a-lambda}, we see that the set of all solutions of the equation $C_a(x) = m$ is given by
\begin{align}
\label{Equation-The-set-S-of-all-solutions-to-the-equation-C-a-(x)=m}
\cS = \{C_b(\lambda_a) + h \; | \; \text{$b \in A$ with $\deg(b) < \deg(a)$}\}.
\end{align}

Since the polynomial $C_a(x) - m$ is separable and $\bL$ is the splitting field of the polynomial $C_a(x) - m \in \bK[x]$, we deduce that $\bL = \bK(\cS)$ and $\bL$ is a Galois extension of $\bK$. Since $\lambda_a = (C_1(\lambda_a) + h) - (C_0(\lambda_a) + h)$, we see that $\lambda_a \in \bL$. We see that $h = C_0(\lambda) + h \in \bL$, and thus
\begin{align*}
\bL = \bK(\lambda_a, h).
\end{align*}

Let $\Gal(\bL/\bK)$ denote the Galois group of the Galois extension $\bL/\bK$. An element $\psi \in \Gal(\bL/\bK)$ is uniquely determined by its action on $\lambda_a$ and $h$. Hence to each element $\psi \in \Gal(\bL/\bK)$, we can associate a unique ordered pair $(b_{\psi}, u_{\psi}) \in (A/aA)^{\times} \times A/aA$ such that $\psi(\lambda_a) = C_{b_{\psi}}(\lambda_a)$ and $\psi(h) = C_{u_{\psi}}(\lambda_a) + h$. Let $\Psi : \Gal(\bL/\bK) \rightarrow \GL_2(A/aA)$ be the map defined by
\begin{align}
\label{Definition-The-representation-Psi-of-the-Galois-group-of-L-over-K}
\Psi(\psi) = \begin{pmatrix} b_{\psi} & u_{\psi} \\ 0 & 1 \end{pmatrix}
\end{align}
for each $\psi \in \Gal(\bL/\bK)$. Since $C$ is the Carlitz module, it is not difficult to see that for any elements $\sigma, \psi \in \Gal(\bL/\bK)$, the action of $\psi \circ \sigma$ on $\lambda_a$ and $h$ are given by
\begin{align*}
(\psi\circ\sigma)(\lambda_a) = C_{b_{\psi}b_{\sigma}}(\lambda_a) \quad \text{and} \quad (\psi\circ\sigma)(h) = C_{b_{\psi}u_{\sigma} + u_{\psi}}(\lambda_a) + h,
\end{align*}
where $(b_{\sigma}, u_{\sigma}), (b_{\psi}, u_{\psi}) \in (A/aA)^{\times} \times A/aA$ are uniquely determined by $\sigma, \psi$, respectively in the same way as above. Hence $\Psi$ is an injective homomorphism.

Let $\Sigma$ be the image of $\Gal(\bL/\bK)$ under $\Psi$. Set
\begin{align}
\label{Equation-The-set-E-in-GL2-A/aA}
\cE = \left\{\begin{pmatrix} \star & \star \\ 0 & 1 \end{pmatrix} \in \GL_2(A/aA) \right\}
\end{align}
and
\begin{align}
\label{Equation-The-set-M-in-GL2-A/aA}
\cM = \left\{\begin{pmatrix} 1 & \star \\ 0 & 1 \end{pmatrix} \in \GL_2(A/aA) \right\}.
\end{align}
We see that $\Sigma$ is a subgroup of $\cE$ and $\cE/\cM \cong (A/aA)^{\times}$. Hence we obtain the embedding
\begin{align}
\label{Equation-The-first-form-of-the-embedding}
\Sigma/\Sigma \cap \cM \hookrightarrow \cE/\cM \cong (A/aA)^{\times}.
\end{align}

We contend that $\Sigma \cap \cM = \left\{\begin{pmatrix} 1 & 0 \\ 0 & 1 \end{pmatrix} \right\}$. Indeed, let $M = \begin{pmatrix} 1 & u \\ 0 & 1 \end{pmatrix}$ be an arbitrary element in $\Sigma \cap \cM$ for some element $u \in A/aA$. Then there exists an element $\psi \in \Gal(\bL/\bK)$ such that $\Psi(\psi) = M$. Let $(b_{\psi}, u_{\psi}) \in (A/aA)^{\times} \times (A/aA)$ be the ordered pair uniquely determined by
\begin{align}
\label{Equation-The-action-of-psi-on-lambda-a}
\psi(\lambda_a) = C_{b_{\psi}}(\lambda_a)
\end{align}
and
\begin{align}
\label{Equation-The-action-of-psi-on-h}
\psi(h) = C_{u_{\psi}}(\lambda_a) + h.
\end{align}
By the definition of $\Psi$, we see that
\begin{align*}
\begin{pmatrix} 1 & u \\ 0 & 1 \end{pmatrix} = M = \Psi(\psi) = \begin{pmatrix} b_{\psi} & u_{\psi} \\ 0 & 1 \end{pmatrix},
\end{align*}
and it thus follows that
\begin{align}
\label{Equation-The-value-of-b-psi}
b_{\psi} = 1
\end{align}
and
\begin{align}
\label{Equation-The-value-of-u-psi}
u_{\psi} = u.
\end{align}

By assumption, we know that there exists a finite set of primes of $\bK$, say, $\cP$ such that $C_a(x) = m$ has a solution $x_{\wp}$ in the completion $\bK_{\wp}$ for every prime $\wp$ of $\bK$ with $\wp \not\in \cP$.

By the Tchebotarev density theorem (see \cite[Theorem {\bf9.13A}]{Rosen}), there exist infinitely many primes $\wp$ in $\bK$ such that $\wp$ is unramified in $\bL$ and $\psi$ is equal to the Frobenius automorphism $(\fB, \bL/\bK)$ for some prime $\fB$ in $\bL$ lying above $\wp$. Since the set $\cP$ is finite, one can choose a prime $\wp$ in $\bK$ with $\wp \not\in \cP$ such that $\wp$ is unramified in $\bL$ and $\psi = (\fB, \bL/\bK)$ for some prime $\fB$ in $\bL$ lying above $\wp$.

Let $\cD(\fB/\wp)$ be the decomposition group of $\fB$ over $\wp$. It is well-known \cite[Proposition {\bf9.8}]{Neukirch} that the fixed field of $\cD(\fB/\wp)$, or equivalently the decomposition field of $\fB$ over $\wp$, is $\bL \cap \bK_{\wp}$. Since $\wp \not\in \cP$, we know that there exists an element $x_{\wp}$ in $\bK_{\wp}$ such that $C_a(x_{\wp}) = m$. Hence $x_{\wp}$ belongs to $\cS$, where $\cS$ is given by $(\ref{Equation-The-set-S-of-all-solutions-to-the-equation-C-a-(x)=m})$. Thus $x_{\wp} = C_b(\lambda_a) + h$ for some $b \in A$ with $\deg(b) < \deg(a)$. In particular this implies that $x_{\wp} \in \bL \cap \bK_{\wp}$.

Since $\psi = (\fB, \bL/\bK) \in \cD(\fB/\wp)$ and $\bL \cap \bK_{\wp}$ is the fixed field of $\cD(\fB/\wp)$, we deduce that
\begin{align*}
\psi(x_{\wp}) = (\fB, \bL/\bK)(x_{\wp}) = x_{\wp}.
\end{align*}
On the other hand, we deduce from $(\ref{Equation-The-value-of-b-psi})$ and $(\ref{Equation-The-value-of-u-psi})$ that
\begin{align*}
C_b(\lambda_a) + h = x_{\wp} = \psi(x_{\wp}) &= \psi(C_b(\lambda_a) + h) \\
&= \psi(C_b(\lambda_a)) + \psi(h) \\
&= C_b(\psi(\lambda_a)) + C_u(\lambda) + h \\
&= C_b(C_1(\lambda_a)) + C_u(\lambda) + h \\
&= C_b(\lambda_a) + C_u(\lambda) + h,
\end{align*}
and thus $C_u(\lambda_a) = 0$.

If $u \ne 0$, then we deduce from Lemma \ref{Lemma-Description-of-elements-in-C-a-lambda} that $C_u(\lambda_a) \ne C_0(\lambda_a)$, and thus $C_u(\lambda_a) \ne 0$, which is a contradiction. Therefore $u = 0$, and hence
\begin{align}
\label{Equation-The-intersection-between-Sigma-and-cM-is-trivial}
\Sigma \cap \cM = \left\{\begin{pmatrix} 1 & 0 \\ 0 & 1 \end{pmatrix} \right\}.
\end{align}

Let $\I$ be the subgroup of $\Gal(\bL/\bK)$ that corresponds to the subfield $\bK(\lambda_a)$ of $\bL$. In other words, $\bK(\lambda_a)$ is the fixed field of $\I$. Take any element $\sigma \in \I$. We know that there exists a unique ordered pair $(b_{\sigma}, u_{\sigma}) \in (A/aA)^{\times} \times A/aA$ such that $\sigma(\lambda_a) = C_{b_{\sigma}}(\lambda_a)$ and $\sigma(h) = C_{u_{\sigma}}(\lambda_a) + h$. Since $\sigma$ belongs to $\I$, we know that $C_{b_{\sigma}}(\lambda_a) = \sigma(\lambda_a) = \lambda_a = C_1(\lambda_a)$, and it thus follows from Lemma \ref{Lemma-Description-of-elements-in-C-a-lambda} that $b_{\sigma} = 1$. Hence we deduce that
\begin{align*}
\Psi(\sigma) = \begin{pmatrix} b_{\sigma} & u_{\sigma} \\ 0 & 1 \end{pmatrix} = \begin{pmatrix} 1 & u_{\sigma} \\ 0 & 1 \end{pmatrix},
\end{align*}
and hence $\Psi(\sigma) \in \Sigma \cap \cM$. By $(\ref{Equation-The-intersection-between-Sigma-and-cM-is-trivial})$, we deduce that $\Psi(\sigma) = \begin{pmatrix} 1 & 0 \\ 0 & 1 \end{pmatrix}$, and thus $u_{\sigma} = 0$. Hence $\sigma$ is the identity map, and therefore $\I =\{\text{id}_{\bL}\}$. Thus by Galois theory, we deduce that $\bL = \bK(\lambda_a)$, which proves our contention.

\end{proof}

\begin{remark}
\label{Remark-The-conditions-can-be-simplified-to-a-certain-positive-density-of-primes}

One of the key ingredients in proving Theorem \ref{Theorem-The-description-of-the-splitting-field-of-the-Carlitz-equation} is the use of Tchebotarev density theorem to prove that $\Sigma \cap \cM = \left\{\begin{pmatrix} 1 & 0 \\ 0 & 1 \end{pmatrix}\right\}$. For an element $\psi \in \Gal(\bL/\bK)$, if we use the precise formula for the Dirichlet density of all unramified primes $\fp$ in $\bK$ such that the set $\left(\fp, \bL/\bK\right)$ of all Frobenius automorphisms equals the conjugacy class of $\psi$ in $\Gal(\bL/\bK)$, then we obtain a stronger result than Theorem \ref{Theorem-The-description-of-the-splitting-field-of-the-Carlitz-equation}. We will shortly prove a refinement of Theorem \ref{Theorem-The-description-of-the-splitting-field-of-the-Carlitz-equation} using the Dirichlet density that leads us to the same conclusion as Theorem \ref{Theorem-The-description-of-the-splitting-field-of-the-Carlitz-equation} with less \textit{restrictive} assumptions.

\end{remark}

Let us first recall some basic notions and the precise statement of Tchebotarev density theorem as presented in \cite[Chapter {\bf9}]{Rosen}. Let $\bL/\bK$ be a Galois extension of global function fields. Let $\cP_{\bK}$ be the set of all primes in $\bK$, and let $\cA$ be a set of primes of $\bK$. If the limit given by
\begin{align}
\label{Equation-The-limit-in-the-definition-of-the-Dirichlet-density}
\lim_{s \rightarrow 1^{+}}\dfrac{\sum_{\wp \in \cA}{N\wp^{-s}}}{\sum_{\wp \in \cP_{\bK}}{N\wp^{-s}}}
\end{align}
exists, we define the \textit{Dirichlet density of $\cA$} to be the value of the limit given by $(\ref{Equation-The-limit-in-the-definition-of-the-Dirichlet-density})$, and denote it by $\delta(\cA)$. If the limit given by $(\ref{Equation-The-limit-in-the-definition-of-the-Dirichlet-density})$ does not exist, we say that $\cA$ does not have the Dirichlet density.

We now recall the statement of Tchebotarev density theorem.
\begin{theorem}
\label{Theorem-Tchebotarev-density-theorem}
$(\text{see \cite[Theorem {\bf9.13A}]{Rosen}})$

Let $\bL/\bK$ be a Galois extension of global function fields, and let $\Gal(\bL/\bK)$ denote the Galois group. Let $\cP^{\text{un}}_{\bK}$ be the set of primes of $\bK$ that are unramified in $\bL$. Let $\C \subset \Gal(\bL/\bK)$ be a conjugacy class in $\Gal(\bL/\bK)$. Let $\cH$ be the subset of $\cP^{\text{un}}_{\bK}$ defined by
\begin{align*}
\cH = \left\{\wp \in \cP^{\text{un}}_{\bK} \; | \; (\wp, \bL/\bK) = \C \right\}.
\end{align*}
Then
\begin{align*}
\delta(\cH) = \dfrac{\Card(\C)}{\Card(\Gal(\bL/\bK))},
\end{align*}
where $\Card(\cdot)$ denotes the number of elements in a set $(\cdot)$.
\end{theorem}

We now prove a refinement of Theorem \ref{Theorem-The-description-of-the-splitting-field-of-the-Carlitz-equation} in which we only assume that the equation $C_a(x) = m$ is locally solvable at all primes $\wp$ in a set of the Dirichlet density greater than $1 - 1/\Card(\Gal(\bL/\bK))$.

\begin{theorem}
\label{Theorem-A-refinement-of-the-theorem-about-the-splitting-field-of-C-a-x-equals-m}

We maintain the same notation as in Theorem \ref{Theorem-The-description-of-the-splitting-field-of-the-Carlitz-equation}. Let $\cH$ be a set of primes in $\bK$ such that the Dirichlet density of $\cH$ satisfies
\begin{align}
\label{Equation-The-lower-bound-of-the-Dirichlet-density-of-H}
\delta(\cH) > 1 - \dfrac{1}{\Card(\Gal(\bL/\bK))}.
\end{align}
Assume that the equation $C_a(x) = m$ has a solution $x_{\wp}$ in the completion $\bK_{\wp}$ for each prime $\wp \in \cH$. Then $\bL = \bK(\lambda_a)$.

\end{theorem}

\begin{remark}
\label{Remark-The-absolute-lower-bound-for-the Dirichlet-density-of-H}

We can replace the lower bound of $\delta(\cH)$ in $(\ref{Equation-The-lower-bound-of-the-Dirichlet-density-of-H})$ by an absolute lower bound that only depends on $a$. Indeed, following the first part of the proof of Theorem \ref{Theorem-The-description-of-the-splitting-field-of-the-Carlitz-equation}, we know that $\bL = \bK(\lambda_a, h)$. Note that this follows directly from the fact that $\bL$ is the splitting field of the polynomial $C_a(x) - m \in \bK[x]$ over $\bK$. By \cite[Theorem {\bf12.3.16}]{Villa-Salvador}, we know that $\lambda_a$ is a root of the $a$-th cyclotomic polynomial $\Phi_a(x) \in k[x] \subset \bK[x]$. We also know from \cite[Theorem {\bf12.3.16}]{Villa-Salvador} that the degree of $\Phi_a(x)$ is $\phi(a)$, where $\phi(a)$ is the number of nonzero polynomials in $A$ of degree less than $\deg(a)$ and relatively prime to $a$. (For a precise formula for $\phi(a)$, see \cite[Proposition {\bf1.7}]{Rosen}.) Thus
\begin{align*}
[\bK(\lambda_a) : \bK] \le \phi(a).
\end{align*}
Since the polynomial $C_a(x) - m \in \bK[x]$ is of degree $q^{\deg(a)}$ and $C_a(h) = m$, we deduce that
\begin{align*}
[\bL : \bK(\lambda_a)] = [\bK(\lambda_a)(h) : \bK(\lambda_a)] \le \deg(C_a(x) - m) = q^{\deg(a)},
\end{align*}
and therefore
\begin{align*}
\Card(\Gal(\bL/\bK)) = [\bL : \bK] = [\bL : \bK(\lambda_a)][\bK(\lambda_a) : \bK] \le \phi(a)q^{\deg(a)}.
\end{align*}
Thus
\begin{align*}
1 - \dfrac{1}{\Card(\Gal(\bL/\bK))} \le 1 - \dfrac{1}{\phi(a)q^{\deg(a)}}.
\end{align*}
We see from the above inequality that if $\delta(\cH) > 1 - \dfrac{1}{\phi(a)q^{\deg(a)}}$, then the condition $(\ref{Equation-The-lower-bound-of-the-Dirichlet-density-of-H})$ follows trivially.

\end{remark}

\begin{proof}[Proof of Theorem \ref{Theorem-A-refinement-of-the-theorem-about-the-splitting-field-of-C-a-x-equals-m}]

We maintain the same notation as in the proof of Theorem \ref{Theorem-The-description-of-the-splitting-field-of-the-Carlitz-equation}. As noted in Remark \ref{Remark-The-conditions-can-be-simplified-to-a-certain-positive-density-of-primes}, a crucial ingredient in the proof of Theorem \ref{Theorem-The-description-of-the-splitting-field-of-the-Carlitz-equation} is the use of Tchebotarev density theorem to prove that
\begin{align}
\label{Equation-The-intersection-of-Sigma-and-M-is-trivial-in-the-refinement-theorem}
\Sigma \cap \cM = \left\{\begin{pmatrix} 1 & 0 \\ 0 & 1 \end{pmatrix}\right\}.
\end{align}
Looking closely at the proof of Theorem \ref{Theorem-The-description-of-the-splitting-field-of-the-Carlitz-equation}, we see that once we establish $(\ref{Equation-The-intersection-of-Sigma-and-M-is-trivial-in-the-refinement-theorem})$, then it follows immediately that $\bL = \bK(\lambda_a)$. Hence it suffices to verify $(\ref{Equation-The-intersection-of-Sigma-and-M-is-trivial-in-the-refinement-theorem})$.

Take any element $M = \begin{pmatrix} 1 & u \\ 0 & 1 \end{pmatrix} \in \Sigma \cap \cM$ for some $u \in A/aA$. There exists an element $\psi \in \Gal(\bL/\bK)$ such that $\Psi(\psi) = M = \begin{pmatrix} 1 & u \\ 0 & 1 \end{pmatrix}$, where $\Psi$ is the injective homomorphism defined by $(\ref{Definition-The-representation-Psi-of-the-Galois-group-of-L-over-K})$. Since $\Psi$ is injective, we see from the definition of $\Psi$ that $\psi(\lambda) = C_1(\lambda_a) = \lambda_a$ and $\psi(h) = C_u(\lambda_a) + h$.

Let $\C_{\psi}$ be the conjugacy class of $\psi$ in $\Gal(\bL/\bK)$. Since $\psi \in \C_{\psi}$, we see that $\Card(\C_{\psi}) \ge 1$. Let $\cB_{\psi}$ be the set defined by
\begin{align*}
\cB_{\psi} = \left\{\wp \in \cP^{\text{un}}_{\bK} \; | \; (\wp, \bL/\bK) = \C_{\psi}  \right\},
\end{align*}
where $\cP^{\text{un}}_{\bK}$ is the set of all primes of $\bK$ that are unramified in $\bL$. We prove that $\cH \cap \cB_{\psi} \ne \emptyset$. Assume the contrary, that is, $\cH \cap \cB_{\psi} = \emptyset$. Then
\begin{align}
\label{Equation-The-indentity-between-Dirichlet-densities}
\delta(\cH \cup \cB_{\psi}) = \delta(\cH) + \delta(\cB_{\psi}).
\end{align}
On the one hand, we see that $\delta(\cH \cup \cB_{\psi}) \le \delta(\cP_{\bK}) = 1$, where $\cP_{\bK}$ denotes the set of all primes in $\bK$. On the other hand, since $\Card(\C_{\psi}) \ge 1$, we deduce from Theorem \ref{Theorem-Tchebotarev-density-theorem} that
\begin{align*}
\delta(\cH) + \delta(\cB_{\psi}) &= \delta(\cH) + \dfrac{\Card(\C_{\psi})}{\Card(\Gal(\bL/\bK))} \\
&> \left(1 - \dfrac{1}{\Card(\Gal(\bL/\bK))}\right) + \dfrac{1}{\Card(\Gal(\bL/\bK))} \\
&= 1 \ge \delta(\cH \cup \cB_{\psi}),
\end{align*}
which is a contradiction to $(\ref{Equation-The-indentity-between-Dirichlet-densities})$. Therefore $\cH \cap \cB_{\psi} \ne \emptyset$, and it thus follows that there exists a prime $\fq$ in $\bK$ such that $\fq \in \cH$, $\fq$ is unramified in $\bL$ and the Frobenius automorphism $(\fB/\fq, \bL/\bK)$ equals $\psi$ for some prime $\fB$ in $\bL$ lying above $\fq$.

By assumption, we know that there exists an element $x_{\fq}$ in the completion $\bK_{\fq}$ of $\bK$ such that $C_a(x_{\fq}) = m$. Repeating the same arguments as in the proof of Theorem \ref{Theorem-The-description-of-the-splitting-field-of-the-Carlitz-equation}, we see that
\begin{align*}
x_{\fq} = (\fB/\fq, \bL/\bK)(x_{\fq}) = \psi(x_{\fq}).
\end{align*}
From the above equation, one can follow the arguments in the proof of Theorem \ref{Theorem-The-description-of-the-splitting-field-of-the-Carlitz-equation} to prove that $C_u(\lambda_a) = 0$, and thus $u = 0$. Therefore $(\ref{Equation-The-intersection-of-Sigma-and-M-is-trivial-in-the-refinement-theorem})$ holds, and as remarked in the first paragraph of this proof, we deduce immediately that $\bL = \bK(\lambda_a)$.

\end{proof}

The following result follows immediately from Remark \ref{Remark-The-absolute-lower-bound-for-the Dirichlet-density-of-H}.
\begin{corollary}
\label{Corollary-The-description-of-the-splitting-field-in-terms-of-the-Dirichlet-density}

We maintain the same notation as in Theorem \ref{Theorem-The-description-of-the-splitting-field-of-the-Carlitz-equation}. Let $\cH$ be a set of primes in $\bK$ such that the Dirichlet density of $\cH$ satisfies
\begin{align}
\label{Equation-The-lower-bound-of-the-Dirichlet-density-of-H-in-a-special-case-of-the-Grunwald-Wang-theorem}
\delta(\cH) > 1 - \dfrac{1}{\phi(a)q^{\deg(a)}}.
\end{align}
Assume that the equation $C_a(x) = m$ has a solution $x_{\wp}$ in the completion $\bK_{\wp}$ for each prime $\wp \in \cH$. Then $\bL = \bK(\lambda_a)$.

\end{corollary}

By the above corollary, we obtain a Carlitz module analogue of the Grunwald--Wang theorem in the case that $\bK$ contains a generator of the $A$-module $\Lambda_a$.
\begin{corollary}
\label{Corollary-A-special-case-of-a-Carlitz-module-analogue-of-the-Grunwald-Wang-theorem}

We maintain the same notation as in Theorem \ref{Theorem-The-description-of-the-splitting-field-of-the-Carlitz-equation}. Assume that $\bK$ contains a generator $\lambda_{\star}$ of the module $\Lambda_a$. Let $\cH$ be a set of primes in $\bK$ such that the Dirichlet density of $\cH$ satisfies
\begin{align*}
\delta(\cH) > 1 - \dfrac{1}{\phi(a)q^{\deg(a)}}.
\end{align*}
Then the equation $C_a(x) = m$ has a solution $x_{\bK}$ in $\bK$ if and only if it has a solution $x_{\wp}$ in the completion $\bK_{\wp}$ of $\bK$ for every prime $\wp \in \cH$.

\end{corollary}

\begin{proof}

If $C_a(x) = m$ has a solution $x_{\bK}$ in $\bK$, then it is obvious that it is locally solvable at every prime of $\bK$.

If $C_a(x) = m$ is locally solvable at every prime in $\cH$, we deduce from Corollary \ref{Corollary-The-description-of-the-splitting-field-in-terms-of-the-Dirichlet-density} that $\bL = \bK(\lambda_a)$. Since $\lambda_{\star}, \lambda_a$ are generators of $\Lambda_a$, it follows that $\bL = \bK(\lambda_a) = \bK(\lambda_{\star}) = \bK$. Since $\bL$ is the splitting field of the polynomial $C_a(x) - m \in \bK[x]$ over $\bK$, we deduce that all solutions to the equation $C_a(x) = m$ belong to $\bK$.

\end{proof}

\begin{remark}

It is known (see \cite[Chapter {\bf12}]{Rosen}) that a generator of the $A$-module $\Lambda_a$ is a function field analogue of a primitive $a$-th root of unity in the number field setting. It is well-known (see \cite{Carlitz}, \cite{Goss}, \cite{Hayes}, \cite{Rosen}, and \cite{Thakur-book}) that the map $x \mapsto C_a(x)$ with $a \in A$ is a Carlitz module analogue of the map $x \mapsto x^a$ with $a \in \bZ$ in the number field setting. Thus Corollary \ref{Corollary-A-special-case-of-a-Carlitz-module-analogue-of-the-Grunwald-Wang-theorem} is a Carlitz module analogue of the classical Grunwald--Wang theorem in the case that $\bK$ contains a generator of the $A$-module $\Lambda_a$. For an account of a similar result in the number field setting, see \cite[Theorem {\bf1.1}, Chapter {\bf VIII}]{Milne-CFT}.

\end{remark}

\section{An analogue of the Grunwald--Wang theorem for certain Galois extensions $\bK/k$}
\label{Section-An-analogue-of-the-Grunwald-Wang-theorem-for-the-intersection-of-K-and-k-lambda-a-being-equal-to-k}

In this section, we prove a Carlitz module analogue of the classical Grunwald--Wang theorem for certain Galois extensions $\bK/k$. The following theorem is the main result in this section.

\begin{theorem}
\label{Theorem-A-Carlitz-module-analogue-of-the-Grunwald-Wang-theorem}

We maintain the same notation as in Theorem \ref{Theorem-The-description-of-the-splitting-field-of-the-Carlitz-equation}. Assume that $\bK$ is a Galois extension of $k$ such that every prime $P$ dividing $a$ is unramified in $\bL$. Let $\cH$ be a set of primes of $\bK$ satisfying the following conditions:
\begin{itemize}

\item [(i)] the Dirichlet density of $\cH$ satisfies
\begin{align*}
\delta(\cH) > 1 - \dfrac{1}{\phi(a)q^{\deg(a)}}; \; \text{and}
\end{align*}

\item [(ii)] $\cH$ contains all primes $\wp$ of $\bK$ that divides $a$.

\end{itemize}
Then the equation $C_a(x) = m$ has a solution $x_{\bK}$ in $\bK$ if and only if it has a solution $x_{\wp}$ in the completion $\bK_{\wp}$ of $\bK$ for every prime $\wp$ of $\bK$ with $\wp \in \cH$.

\end{theorem}

\begin{proof}

If $C_a(x) = m$ has a solution $x_{\bK}$ in $\bK$, then it follows immediately that $C_a(x) = m$ has a solution $x_{\wp}$ in the completion $K_{\wp}$ for every prime $\wp$ in $\bK$.

Suppose that $C_a(x) = m$ has a solution $x_{\wp}$ in the completion $\bK_{\wp}$ for each prime $\wp \in \cH$. We will prove that the equation $C_a(x) = m$ has a solution $x_{\bK} \in \bK$.

By assumption, we know from Corollary \ref{Corollary-The-description-of-the-splitting-field-in-terms-of-the-Dirichlet-density} that $\bL = \bK(\lambda_a)$.

Let $h$ be an element in $k^{\alg}$ such that $C_a(h) = m$, and let $\cS$ be the set defined by $(\ref{Equation-The-set-S-of-all-solutions-to-the-equation-C-a-(x)=m})$ in the proof of Theorem \ref{Theorem-The-description-of-the-splitting-field-of-the-Carlitz-equation}. Since $a$ is a polynomial of positive degree, we can write $a$ in the form
\begin{align}
\label{Equation-The-representation-of-a}
a = \epsilon P_1^{e_1}\cdots P_n^{e_n},
\end{align}
where $\epsilon$ is in $\bF_q^{\times}$, the $P_i$ are distinct monic primes in $A$, and the $e_i$ are positive integers. For each $1 \le i \le n$, set
\begin{align}
\label{Equation-The-equation-of-a-i-dividing-a}
a_i := \dfrac{a}{P_i^{e_i}} \in A,
\end{align}

We prove that for each $1 \le i \le n$, there exists an element $b_i \in A$ with $\deg(b_i) < \deg(a)$ such that $C_{a_i}(C_{b_i}(\lambda_a) + h) \in \bK$. Indeed, take an integer $1 \le i \le n$. Set
\begin{align*}
\Lambda_{P_i^{e_i}} = \left\{\gamma \in k^{\alg} \; | \; C_{{P_i}^{e_i}}(\gamma) = 0 \right\}.
\end{align*}
We know from \cite[Proposition {\bf12.7}]{Rosen} that $k(\Lambda_{P_i^{e_i}})$ is a Galois extension of $k$.

Let $\fq$ be a prime of $\bK$ lying above $(P_i) = P_iA$, and let $\fQ$ be a prime of $\bK(\Lambda_{P_i^{e_i}})$ lying above $\fq$. Let $\fB$ be a prime of $k(\Lambda_{P_i^{e_i}})$ lying below $\fQ$. Then $\fB$ lies above $(P_i)$.
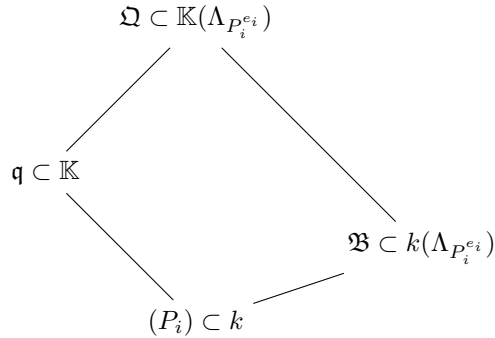
\begin{center}
\begin{tikzpicture}[node distance = 2cm, auto]
\node (k) {$(P_i) \subset k$};
\node (K) [above of=k, left of=k] {$\fq \subset \bK$};
\node (k1) [node distance = 1cm, above of=k, right of=k, right of =k, right of =k] {$\fB \subset k(\Lambda_{P_i^{e_i}})$};
\node (K1) [above of=k, node distance = 4cm] {$\fQ \subset \bK(\Lambda_{P_i^{e_i}})$};
\draw[-] (k) to node {} (K);
\draw[-] (k) to node [swap] {} (k1);
\draw[-] (K) to node {} (K1);
\draw[-] (k1) to node [swap] {} (K1);
\end{tikzpicture}
\end{center}

We prove that $\fq$ is totally ramified in $\bK(\Lambda_{P_i^{e_i}})$. Let $\e(\fQ/\fq)$, $\e(\fq/(P_i))$ be the ramification indices of $\fQ$ over $\fq$ and $\fq$ over $(P_i)$, respectively. Let $\e(\fQ/\fB)$, $\e(\fB/(P_i))$ be the ramification indices of $\fQ$ over $\fB$ and $\fB$ over $(P_i)$, respectively. Let $\e(\fQ/(P_i))$ be the ramification index of $\fQ$ over $(P_i)$. We know from \cite[Proposition {\bf12.7}]{Rosen} that the prime ideal $(P_i) = P_iA$ is totally ramified in the extension $k(\Lambda_{P_i^{e_i}})/k$, and thus
\begin{align*}
\e(\fB/(P_i)) = [k(\Lambda_{P_i^{e_i}}) : k] = \phi(P_i^{e_i}),
\end{align*}
where $\phi(\cdot)$ denotes the function field analogue of the Euler $\phi$-function (see, for example, \cite[Proposition {\bf1.7}]{Rosen} for the definition of $\phi(\cdot)$). By assumption, we know that $(P_i)$ is unramified in $\bK$, and thus $\e(\fq/(P_i)) = 1$. Since $K(\Lambda_{P_i^{e_i}}) = K\cdot k(\Lambda_{P_i^{e_i}})$, it thus follows from \cite[Lemma {\bf4.6.3}]{Koch} that $\fB$ is unramified in $\bK(\Lambda_{P_i^{e_i}})$. Hence $\e(\fQ/\fB) = 1$, and we deduce that
\begin{align}
\label{Equation-The-ramification-index-of-fQ/fq}
\e(\fQ/\fq) = \e(\fQ/\fq)\e(\fq/(P_i)) = \e(\fQ/(P_i)) = \e(\fQ/\fB)\e(\fB/(P_i)) = \e(\fB/(P_i)) = \phi(P_i^{e_i}).
\end{align}
We contend that the degree of $\bK(\Lambda_{P_i^{e_i}})$ over $\bK$ is less than or equal to $\phi(P_i^{e_i})$. Indeed, let $\chi$ be a primitive generator of $\Lambda_{P_i^{e_i}}$. We know from \cite[Theorem {\bf12.3.16}]{Villa-Salvador} that $\chi$ is a root of the $P_i^{e_i}$--th cyclotomic polynomial $\Phi_{P_i^{e_i}}(x) \in k[x] \subset \bK[x]$ of degree $\phi(P_i^{e_i})$. Hence we deduce that
\begin{align*}
[\bK(\Lambda_{P_i^{e_i}}) : \bK] = [\bK(\chi) : \bK] \le \deg\left(\Phi_{P_i^{e_i}}(x)\right) = \phi(P_i^{e_i}).
\end{align*}
Therefore it follows from $(\ref{Equation-The-ramification-index-of-fQ/fq})$ that
\begin{align*}
\phi(P_i^{e_i}) = \e(\fQ/\fq) \le [\bK(\Lambda_{P_i^{e_i}}) : \bK] \le \phi(P_i^{e_i}),
\end{align*}
and hence
\begin{align*}
\e(\fQ/\fq) = [\bK(\Lambda_{P_i^{e_i}}) : \bK] = \phi(P_i^{e_i}).
\end{align*}
Thus $\fq$ is totally ramified in $\bK(\Lambda_{P_i^{e_i}})$.

We now consider the equation
\begin{align}
\label{Equation-The-equation-C-P^e-x-m}
C_{P_i^{e_i}}(z) = m.
\end{align}
Let $\bL_i$ be the splitting field of the polynomial $C_{P_i^{e_i}}(z) - m \in \bK[z]$ over $\bK$. Since the polynomial $C_{P_i^{e_i}}(z) - m$ is separable, we see that $\bL_i$ is a Galois extension of $\bK$.

We prove that the equation given by $(\ref{Equation-The-equation-C-P^e-x-m})$ has a solution $z_{\wp}$ in the completion $\bK_{\wp}$ for each prime $\wp \in \cH$. Indeed take any prime $\wp \in \cH$. We know that there exists an element $x_{\wp}$ in $\bK_{\wp}$ such that $C_a(x_{\wp}) = m$. Since $a = P_i^{e_i}a_i$, we deduce that
\begin{align}
\label{Equation-The-equation-C-P^e(x)-equals-m-is-locally-solvable}
C_{P_i^{e_i}}(C_{a_i}(x_{\wp})) = C_{P_i^{e_i}a_i}(x_{\wp}) = C_a(x_{\wp}) = m.
\end{align}
Since $C_{a_i}(x) \in k[x] \subset \bK[x]$ and $\bK \subset \bK_{\wp}$, we see that $C_{a_i}(x_{\wp}) \in \bK_{\wp}$. Upon letting $z_{\wp} = C_{a_i}(x_{\wp}) \in \bK_{\wp}$, we see from $(\ref{Equation-The-equation-C-P^e(x)-equals-m-is-locally-solvable})$ that the equation $C_{P_i^{e_i}}(z) = m$ has a solution $z_{\wp}$ in $\bK_{\wp}$. Therefore the equation $C_{P^e}(z) = m$ has a solution $z_{\wp}$ in the completion $\bK_{\wp}$ for each prime $\wp \in \cH$. Applying Corollary \ref{Corollary-The-description-of-the-splitting-field-in-terms-of-the-Dirichlet-density} for the equation $C_{P_i^{e_i}}(z) = m$ with $P_i^{e_i}, m, \bL_i, \chi$ in the roles of $a, m, \bL, \lambda_a$, respectively, where we recall that $\chi$ is a primitive generator of $\Lambda_{P_i^{e_i}}$, we deduce that $\bL_i = \bK(\chi) = \bK(\Lambda_{P_i^{e_i}})$, and thus
\begin{align}
\label{Equation-The-Galois-group-of-L-i-over-k}
\Gal(\bL_i/\bK) = \Gal(\bK(\Lambda_{P_i^{e_i}})/\bK).
\end{align}

Recall that we have shown that $\fq$ is totally ramified in $ \bK(\Lambda_{P_i^{e_i}})$ and the prime $\fQ$ lies above $\fq$. Let $\cD(\fQ/\fq) \le \Gal(\bK(\Lambda_{P_i^{e_i}})/\bK)$ be the decomposition group of $\fQ$ over $\fq$. Since $\fq$ is totally ramified in $\bK(\Lambda_{P_i^{e_i}})$, we deduce from $(\ref{Equation-The-Galois-group-of-L-i-over-k})$ that
\begin{align}
\label{Equation-The-decomposition-group-of-fQ-over-fq}
\cD(\fQ/\fq) = \Gal(\bK(\Lambda_{P_i^{e_i}})/\bK) = \Gal(\bL_i/\bK).
\end{align}

By assumption $(ii)$ in the statement of Theorem \ref{Theorem-A-Carlitz-module-analogue-of-the-Grunwald-Wang-theorem}, there exists an element $x_{\fq}$ in the completion $\bK_{\fq}$ such that $C_a(x_{\fq}) = m$. Hence $x_{\fq}$ belongs to the set $\cS$, where $\cS$ is the set given by $(\ref{Equation-The-set-S-of-all-solutions-to-the-equation-C-a-(x)=m})$ in the proof of Theorem \ref{Theorem-The-description-of-the-splitting-field-of-the-Carlitz-equation}. Thus $x_{\fq} = C_{b_i}(\lambda_a) + h$ for some element $b_i \in A$ with $\deg(b_i) < \deg(a)$. We see that
\begin{align*}
C_{P_i^e}(C_{a_i}(x_{\fq})) = C_{P_i^{e_i}a_i}(x_{\fq}) = C_a(x_{\fq}) = m,
\end{align*}
and thus $C_{a_i}(x_{\fq})$ is a root of the polynomial $C_{P_i^{e_i}}(z) - m \in \bK[z]$. Since $\bL_i = \bK(\Lambda_{P_i^{e_i}})$ is the splitting field of the polynomial $C_{P_i^e}(z) - m$ over $\bK$, we deduce that
\begin{align}
\label{Equation-The-field-of-definition-of-C-ai-x-fq}
m_i := C_{a_i}(x_{\fq}) = C_{a_i}(C_{b_i}(\lambda_a) + h) \in \bK_{\fq} \cap \bL_i.
\end{align}

Since $\bK_{\fq} \cap \bL_i$ is the fixed field of $\cD(\fQ/\fq)$ (see \cite[Proposition {\bf9.8}]{Neukirch}), we see from $(\ref{Equation-The-field-of-definition-of-C-ai-x-fq})$ that $\psi(m_i) = m_i$ for every $\psi \in \cD(\fQ/\fq)$. By $(\ref{Equation-The-decomposition-group-of-fQ-over-fq})$, we know that $\cD(\fQ/\fq) = \Gal(\bL_i/\bK)$, and hence $m_i$ belongs to $\bK$.

In summary, we have shown that each $1 \le i \le n$, there exists an element $b_i \in A$ with $\deg(b_i) < \deg(a)$ such that $m_i = C_{a_i}(C_{b_i}(\lambda_a) + h) \in \bK$.

Since the elements $P_i^{e_i}$ are pairwise coprime in the principal ideal domain $A$, it follows from the Chinese Remainder Theorem that there exists an element $\epsilon$ in $A$ such that
\begin{align}
\label{Equation-The-element-epsilon-in-A}
\epsilon \equiv b_i \pmod{P_i^{e_i}}
\end{align}
for each $1 \le i \le n$. Set
\begin{align}
\label{Equation-The-equation-of-x-bK}
x_{\bK} := C_{\epsilon}(\lambda_a) + h.
\end{align}
We prove that $x_{\bK}$ belongs to $\bK$, and is a solution to the equation $C_a(x) = m$. We first show that $C_{a_i}(x_{\bK}) = m_i$ for each $1 \le i \le n$.

Take any integer $1 \le i \le n$. By $(\ref{Equation-The-element-epsilon-in-A})$, there exists an element $\delta_i \in A$ such that $\epsilon = b_i + \delta_iP_i^{e_i}$. It thus follows from $(\ref{Equation-The-equation-of-x-bK})$ and $(\ref{Equation-The-equation-of-a-i-dividing-a})$ that
\begin{align*}
C_{a_i}(x_{\bK}) &= C_{a_i}(C_{\epsilon}(\lambda_a) + h) \\
&= C_{a_i}(C_{b_i + \delta_iP_i^{e_i}}(\lambda_a) + h) \\
&= C_{a_i}(C_{b_i}(\lambda_a) + h) + C_{a_i}(C_{\delta_iP_i^{e_i}}(\lambda_a))\\
&= m_i + C_{\delta_i a_iP_i^{e_i}}(\lambda_a) \\
&= m_i + C_{\delta_i a}(\lambda_a) \\
&= m_i + C_{\delta_i}(C_a(\lambda_a)) \\
&= m_i.
\end{align*}
In particular, this implies that $C_{a_i}(x_{\bK}) = m_i \in \bK$ for every $1 \le i \le n$.

By $(\ref{Equation-The-equation-of-a-i-dividing-a})$, we deduce that $\gcd(a_1, a_2, \ldots, a_n) = 1$, and thus there exist elements $\upsilon_1, \upsilon_2, \ldots, \upsilon_n$ in $A$ such that
\begin{align*}
\upsilon_1 a_1 + \upsilon_2 a_2 + \ldots + \upsilon_n a_n = 1.
\end{align*}
Since $C_{\upsilon_i}(x) \in k[x] \subset \bK[x]$ for every $1 \le i \le n$, we deduce that
\begin{align*}
x_{\bK} = C_1(x_{\bK}) = C_{\upsilon_1 a_1 + \upsilon_2 a_2 + \cdots + \upsilon_n a_n}(x_{\bK}) = C_{\upsilon_1}(C_{a_1}(x_{\bK})) + \cdots + C_{\upsilon_n}(C_{a_n}(x_{\bK})) \in \bK.
\end{align*}
Furthermore since $m_1 = C_{a_1}(C_{b_1}(\lambda_a) + h)$ is a solution to the equation $C_{P_1^{e_1}}(z) = m$, we deduce that
\begin{align*}
C_{a}(x_{\bK}) = C_{P_1^{e_1}a_1}(x_{\bK}) = C_{P_1^{e_1}}(C_{a_1}(x_{\bK})) = C_{P_1^{e_1}}(m_1) = m,
\end{align*}
which proves that the equation $C_a(x) = m$ has a solution $x_{\bK} \in \bK$. Thus our contention follows.

\end{proof}

Let $\cP_{\bK}$ be the set of all primes in $\bK$. If $\cH$ is a set of primes of $\bK$ such that $\cP_{\bK}\setminus \cH$ is finite and $\cH$ contains all primes of $\bK$ that divide $a$, then $\cH$ is of Dirichlet density $1$. Hence we obtain the following result that is a special case of Theorem \ref{Theorem-A-Carlitz-module-analogue-of-the-Grunwald-Wang-theorem}.
\begin{corollary}

We maintain the same notation as in Theorem \ref{Theorem-The-description-of-the-splitting-field-of-the-Carlitz-equation}. Assume that $\bK$ is a Galois extension of $k$ such that every prime $P$ dividing $a$ is unramified in $\bK$. Then the equation $C_a(x) = m$ has a solution $x_{\bK}$ in $\bK$ if and only if it has a solution $x_{\wp}$ in the completion $\bK_{\wp}$ of $\bK$ for all primes $\wp \in \cH$, where $\cH$ is a set of primes of $\bK$ such that $\cP_{\bK}\setminus \cH$ is finite and $\cH$ contains all primes of $\bK$ that divide $a$.

\end{corollary}

\end{document}